\documentclass[a4paper,12pt]{article}
\usepackage[centertags]{amsmath}
\usepackage{amsfonts}
\usepackage{amssymb}
\usepackage{amsthm}
\usepackage{dsfont}
\usepackage{tikz, subfigure}
\usepackage{verbatim}

\newtheorem{theorem}{Theorem}[section]
\newtheorem{lemma}{Lemma}[section]
\newtheorem{cor}{Corollary}[section]
\newtheorem{prop}{Proposition}[section]

\begin{document}
\title{The Scenery Flow for Self-Affine Measures.}
\author{Tom Kempton}
\maketitle
\abstract{\noindent We describe the scaling scenery associated to Bernoulli measures supported on separated self-affine sets under the condition that certain projections of the measure are absolutely continuous.}

\setlength{\parskip}{0.3cm} \setlength{\parindent}{0cm}

\section{Introduction}

The scenery flow is an extremely useful tool for studying fractal sets and measures. Recently several long standing conjectures in fractal geometry have been resolved using the scenery flow. In particular, Furstenberg proved a dimension conservation result for uniformly scaling measures which generate ergodic fractal distributions and Hochman and Shmerkin gave conditions under which every projection of a fracal measure $\mu$ has dimension equal to $\min\{\dim_H(\mu),1\}$, \cite{FurstenbergDC, HochShmerProj}. The scenery flow has also been used to prove several important results in geometric measure theory, \cite{KSS,OrponenDistance, SSS1}. For this reason, much attention has been given recently to the problem of understanding the scenery flow for various classes of fractal measures, and in particular the question of whether they are uniformly scaling and whether they generate ergodic fractal distributions.

The scenery flow for non-overlapping self-similar and self-conformal measures is well understood, \cite{BF1, BF2, FPConformalScaling, Patzschke}. In the self-affine setting, the scenery flow has previously been studied for measures on Bedford-McMullen carpets, \cite{FFS, Almarza}, and Hochman asked whether it can be understood more generally \cite{HochmanDynamics}. In this article we study the scenery flow for a wide class of self-affine measures which satisfy a cone condition and a projection condition, given later.

There is much interesting dynamics associated to self-affine sets and measures which is not present in the self-similar case. In particular, iterated function systems defining a self-affine set give rise to further iterated function systems on projective space which describe the way in which straight lines through the origin are mapped onto each other by affine maps. This second iterated function system defines the Furstenberg measure on projective space, which is crucial to understanding self-affine measures. Recently formulae for the Hausdorff dimension of a self-affine set were given in terms of the dimension of projections of the self-affine measure in typical directions chosen according to the Furstenberg measure, \cite{ Barany, FalKemp}. Dimension theory for self-affine sets is an extremely active topic of research, see for example the survey papers \cite{PesinChen, FalconerAffineSurvey}, and yet a general theory does not yet exist. We hope that as the understanding of the scenery flow for self-affine sets becomes more developed, a general theory of dimension for self-affine sets may emerge.

In this article we build on our work with Falconer on the dynamics of self-affine sets, \cite{FalKemp}, to describe the scenery flow for self-affine measures associated to strictly positive matrices under the condition that projections of the self-affine measure in typical directions for the Furstenberg measure are absolutely continuous. This projection condition holds typically on large parts of parameter space, \cite{BPS}, and holds everywhere for some open sets in parameter space \cite{FalKemp}. Very recently the scenery flow for self-affine sets rather than measures was studied in \cite{KHR}. Additionally, we study the scenery flow for slices through self-affine measures without assuming any condition on projections.

\subsection{The Scenery Flow}

Let $\mathcal M$ denote the space of Borel probability measures $\mu$ supported on the unit disk $X$ with $0\in supp(\mu)$. Let $d$ denote the Prokhorov metric on $\mathcal M$, given by
\[
d(\mu,\nu):=\inf\{\epsilon: \mu(A)\leq \nu(A_{\epsilon})+\epsilon, \nu(A)\leq \mu(A_{\epsilon})+\epsilon \text{ for all Borel sets }A\}
\]
where $A_{\epsilon}:=\{ x\in\mathbb R^2: d(x,y)<\epsilon$ for some $y\in A\}$. The Prokhorov metric metrises the weak$^*$ topology.

Let $B(x,r)$ denote the ball of radius $r$ centred at $x\in\mathbb R^2$. Given $\mu\in\mathcal M$ we let $S_t(\mu)$ denote the measure $\mu|_{B(0,e^{-t})}$, normalised to have mass $1$ and mapped onto the unit disk by the dilation map $ x\to e^t x$ for $x\in\mathbb R^2$. Note that $S_{t+s}(\mu)=S_t(S_s(\mu))$ and so $S$ is a well defined flow on the space $\mathcal M$.


We refer to $S_t$-invariant measures $P$ on the space $\mathcal M$ as distributions. Applying ergodic theory to the system $(\mathcal M, P, S_t)$ turns out to be extremely useful in geometric measure theory and the study of fractals.

The flow $S_t$ describes the process of zooming in on the measure $\mu$ around the origin. If we are interested in zooming in on some other point $x\in\mathbb R^2$ we can first apply the map $T_x$ given by $T_x(y):=y-x$ and then apply $S_t$ to the resulting measure. For shorthand, we let $S_{t,x}(\mu):=S_t\circ T_x(\mu)$.

We let
\[
<\mu>_{T,x}:=\frac{1}{T}\int_0^T\delta_{S_{t,x}(\mu)} dt
\]
be called the scenery distribution of $\mu$ at $x$ up to time $T$. $<\mu>_{T,x}$ gives mass
\[
\frac{1}{T}\int_0^T \chi_A(S_{t,x}(\mu))dt
\]
to Borel subsets $A$ of $\mathcal M$. If $<\mu>_{T,x}\to P$ in the weak$^*$ topology as $t\to\infty$ we say that $\mu$ {\it generates} $P$ at $x$. The measure $\mu$ is known as a {\it uniformly scaling measure} if it generates the same distribution $P$ at $\mu$-almost every point $x$, in which case we say $\mu$ generates $P$.

The quasi-Palm property is a property of $S_t$-invariant distributions which describes a kind of translation invariance. We say that a distribution $P$ is {\it quasi-Palm} if, for a subset $A$ of $\mathcal M$ we have $P(A)=0$ if and only if the measures $S_{0,x}(\mu)$ obtained by choosing $\mu$ according to $P$, choosing $x$ according to $\mu$ are almost surely not in $A$. See \cite{HochmanDynamics} for a more full discussion of the quasi-Palm property.

An {\it ergodic fractal distribution} is an $S_t$ invariant, ergodic probability distribution on $\mathcal M$ which is quasi-Palm. The best case scenario for inferring properties of measures $\mu$ from the distributions they generate is that $\mu$ is a uniformly scaling measure generating an ergodic fractal distribution, this will not be the case for the class of self-affine measures which we consider because of a rotation element which depends upon the point around which we are zooming in, but if were to disregard this rotational effect then the generated measures would indeed be uniformly scaling.



\section{A First Example}\label{PrzSection}
We begin by studying an example of a self-affine measure which demonstrates the extra difficulties associated with studying the scenery flow for self-affine, as opposed to self-similar, measures. This example also demonstrates how, when certain relevant projections of the self-affine measure are absolutely continuous, these extra difficulties can be overcome.

The examples we study are a class of self-affine carpets first studied by Przytycki and Urbanski \cite{PrzUrb}. These carpets have rather less structure than Bedford-McMullen carpets, and so previous techniques of \cite{FFS, Almarza} cannot be applied. For this example we scale along squares rather than balls. 

For $\lambda\in(\frac{1}{2},1)$ consider the self affine set $E_{\lambda}\subset [0,1]^2$ which is the attractor of the iterated function system given by contractions
\[
T_0(x,y)=\left(\lambda x, \frac{y}{3}\right),~T_1(x,y)=\left(\lambda x+(1-\lambda), \frac{y+2}{3}\right).
\]

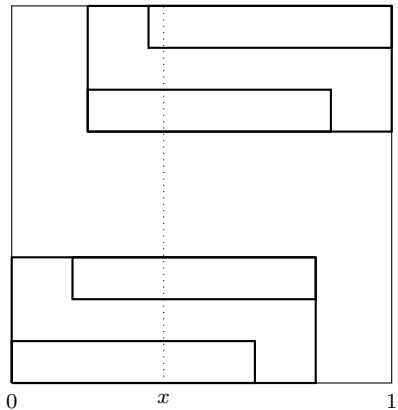
\begin{figure}[h]
\centering
\begin{tikzpicture}[scale=5]
\draw(0,0)node[below]{\scriptsize 0}--(1,0)node[below]{\scriptsize $1$}--(1,1)--(0,1)--(0,0);
\draw[thick](0,0)--(.8,0)--(0.8,0.33333)--(0,0.33333)--(0,0);
\draw[thick](0.2,1)--(0.2,0.66666)--(1,0.66666)--(1,1)--(0.2,1);
\draw[thick](0,0)--(0.64,0)--(0.64,0.11111)--(0,0.11111)--(0,0);
\draw[thick](0.8,0.33333)--(0.16,0.33333)--(0.16,0.22222)--(0.8,0.22222)--(0.8,0.33333);
\draw[thick](1,1)--(0.36,1)--(0.36,0.88888)--(1,0.88888)--(1,1);
\draw[thick](0.2,0.66666)--(0.84,0.66666)--(0.84,0.77777)--(0.2,0.77777)--(0.2,0.66666);
\draw[dotted](0.4,0)node[below]{\scriptsize $x$}--(0.4,1);
\end{tikzpicture}\caption{The first two levels of $E_{0.8}$}
\end{figure}

For $a_1\cdots a_n\in\{0,1\}^n$ we let \[E_{a_1\cdots a_n}:= T_{a_1\cdots a_n}(E_{\lambda})\]
where
\[
T_{a_1\cdots a_n}:=T_{a_1}\circ T_{a_2}\circ\cdots T_{a_n}.
\]
Each point $(x,y)\in E_{\lambda}$ has a unique code $\underline a\in\{0,1\}^{\mathbb N}$ such that $(x,y)\in E_{a_1\cdots a_n}\forall n\in\mathbb N$. We define the map $\pi:\{0,1\}^{\mathbb N}\to E_{\lambda}$ to be the map from a code to the corresponding point $(x,y)\in E_{\lambda}$. Let $\mu$ be the measure which arises from mapping the $(\frac{1}{2},\frac{1}{2})$ Bernoulli measure on $\{0,1\}^{\mathbb N}$ to $E_{\lambda}$ by the coding map $\pi$. 

Now given a point $\pi(\underline a)\in E_{\lambda}$ we let $B(\pi(\underline a), 3^{-n})$ denote the square, centred at $\pi(\underline a)$, of side length $2.3^{-n}$. We further denote $R(\sigma^n(\underline a), (3\lambda)^{-n})$ the rectangle centred at $\pi(\sigma^n(\underline a))$ of height $2$ and width $2.(3\lambda)^{-n}<2$.

We would like to understand the measure $S_{n\log 3}(\mu, \underline a)$ obtained by taking $\frac{\mu|_{B(\pi(\underline a), 3^{-n})}}{\mu(B(\pi(\underline a), 3^{-n}))}$ and linearly rescaling it to live on the square $[-1,1]^2$. If our maps $T_i$ were non-overlapping similarities, this rescaling would be a rather straightforward process, we would just need to apply inverses of our contraction $T_i$ which would scale up the small square to get a large square.

Since our maps $T_i$ are affine contractions but not similarities, we instead need to apply a two step process. Our square $B(\pi(\underline a), 3^{-n})$ intersects precisely one level $n$ rectangle in the construction of $E$, namely the rectangle $E_{a_1\cdots a_n}$. First we apply the map $T_{a_1\cdots a_n}^{-1}$ to $B(\pi(\underline a), 3^{-n})$ to get the rectangle $R(\sigma^n(\underline a), (3\lambda)^{-n})$. Here $\sigma$ denotes the shift map on $\{0,1\}^{\mathbb N}$. The self-affinity relation for $\mu$ gives that $\mu|_{R(\sigma^n(\underline a), (3\lambda)^{-n})}$ is an affine copy of the measure $\mu|_{B(\pi(\underline a), 3^{-n})}$.

To complete our process, we need to stretch the rectangle $R(\sigma^n(\underline a), (3\lambda)^{-n})$ horizontally by a factor of $(3\lambda)^n$ and translate the resulting square onto $[-1,1]^2$. Denote by $D(\underline b,n)$ the map which stretches the rectangle $R(\underline b, (3\lambda)^{-n})$ linearly onto $[-1,1]^2$. We have
\begin{equation}\label{Seqex}
S_{t,\underline a}(\mu)= \left(\frac{\mu|_{B(\pi(\underline a), 3^{-n})}}{\mu(B(\pi(\underline a), 3^{-n}))}\right)\circ T_{a_1\cdots a_n}\circ D(\sigma^n(\underline a),n)^{-1}
\end{equation}

There are three key observations which allow us to understand the scenery flow for this example, and for the broader class of self-affine measures considered below.

{\bf Observation 1:} We have
\[
\left(\frac{\mu|_{B(\pi(\underline a), 3^{-n})}}{\mu(B(\pi(\underline a), 3^{-n}))}\right)\circ T_{a_1\cdots a_n}= \dfrac{\mu|_{R(\sigma^n(\underline a), (3\lambda)^{-n})}}{\mu(R(\sigma^n(\underline a), (3\lambda)^{-n}))}
\]

This follows directly from the self-affinity of the measure $\mu$.

{\bf Observation 2:} Suppose that, for $\mu$-almost every $\underline b\in\{0,1\}^{\mathbb N}$, the sequence of measures
\[
\frac{\mu|_{R(\underline b, (3\lambda)^{-n})}}{\mu(R(\underline b, (3\lambda)^{-n}))}\circ D(\underline b,n)^{-1}
\]
on $[-1,1]^2$ converges weak$^*$ to some limit measure $\mu_{\underline b}$ as $n\to\infty$. Then for all $\epsilon>0$ and almost all $\underline a\in\{0,1\}^{\mathbb N}$ there exists a set $A\subset \mathbb N$ with
\[
\lim_{n\to\infty}\frac{1}{n}|A\cap\{1,\cdots,n\}|>1-\epsilon
\]
such that the sequence of measures $S_{n\log 3,\underline a}(\mu)$ restricted to $n\in A$ is weakly-asymptotic to the sequence of measures $\mu_{\sigma^n(\underline a)}$. Hence by the ergodicity of the system $(\{0,1\}^{\mathbb N}, \sigma,\mu)$ we have that $\mu$ is a uniformly scaling measure generating an ergodic fractal distribution. 

This follows immediately from equation \ref{Seqex}. We use Egorov's theorem to turn almost everywhere convergence to $\mu_{\underline b}$ into uniform convergence on a large set of $\underline b\in\{0,1\}^{\mathbb N}$, which in turn allows us to generate the set $A$. 

{\bf Observation 3:} Suppose that the projection of $\mu$ onto the horizontal axis is absolutely continuous. Then the limit measures $\mu_{\underline b}$ of Observation 2 exist for $\mu$-almost every $\underline b$, and hence $\mu$ is a uniformly scaling measure generating an ergodic fractal distribution.

Observation 3 is less straightforward than the previous two. It relies firstly on the fact that one can disintegrate a measure by vertical slicing. Secondly we use Lemma \ref{ACFlow}, given below, which says that the scenery flow converges $\nu$-almost everywhere for measures $\nu$ which are absolutely continuous. The measures $\mu_{\underline b}$ take the form of a vertical slice of $\mu$ through $\underline b$ crossed with Lebesgue measure. The slice measures were described in \cite{KemptonJEMS}. We do not give further justification for observation $3$ here, the corresponding proposition applying to more general self-affine measures is proved later.

An important result on which we rely is the following version of the Lebesgue density theorem.
\begin{lemma}\label{ACFlow}
Given an absolutely continuous measure $\nu$ on $[-1,1]$, for $\nu$-almost every $x$ the scenery flow applied to $\nu$ around $x$ converges to Lebesgue measure.
\end{lemma}

Note that the projected measures of Observation 3 are a well studied family of self-similar measures known as Bernoulli convolutions, which are absolutely continuous for all $\lambda\in(\frac{1}{2},1)$ outside of a family of exceptions which has Hausdorff dimension $0$ \cite{ShmerkinAC}. Thus, combining observations 1,2 and 3, we have the following theorem.

\begin{theorem}\label{PrzTheorem}
For all $\lambda\in(\frac{1}{2},1)$ outside of a set of exceptions of Hausdorff dimension zero, the $\left(\frac{1}{2},\frac{1}{2}\right)$-Bernoulli measures $\mu$ on the fractal $E_{\lambda}$ are uniformly scaling measures which generate an ergodic fractal distribution.
\end{theorem}
A proof of this theorem follows fairly directly from the above three observations. We prefer to regard it as a corollary to the more general Theorem \ref{MainTheorem}. 

The fact that our results for this example hold only for measures on sets $E_{\lambda}$ for which the corresponding Bernoulli convolution is absolutely continuous may seem like a significant restriction, essentially we are restricting to the case that we already understand quite well. However, as one generalises from the carpet like case of this example to more general self-affine sets the absolute continuity of `relevant projections' becomes rather more natural, and the theorems that we prove later can be shown to hold for open sets in parameter space.

\subsection{A Comment on the Projection Condition}
Putting together the three observations above allows one to describe the scenery flow for the measure $\mu$ under the condition that image under vertical projection of $\mu$ is an absolutely continous measure. A similar projection condition is required in the later, more general situation.

One might hope to be able to prove the same results about $\mu$ under the looser projection condition that the vertical projection of $\mu$ is a uniformly scaling measure generating an ergodic fractal distribution, i.e. rather than requiring the convergence of the scenery flow for typical points in the projected measure, one would only require that the scenery flow on the projected measure is asymptotic to an ergodic flow. 

The issue here is that one would have to do consider two ergodic maps simultaneously, the first map $\underline b \to \sigma^n(\underline b)$ governing the way in which the centre point of Observation 1 moves, and the second map doing the time $n\log (3\lambda)$ scenery flow on the vertical projection of $\mu$ around point $\pi(\sigma^n(\underline b))$. We are unable to guarantee that there is no resonance between these two ergodic maps and that the resulting flow generates the ergodic distributions expected. This may be fixable in the specific example of this section, but in the more general setting which follows it appears out of reach for the moment.

\section{Positive Matrices and the Furstenberg Measure}

Let $k\in\mathbb N$ and for each $i\in\{1,\cdots k\}$ let $A_i$ be a real valued $2\times 2$ matrix of norm less than one. We also assume that each entry of each matrix $A_i$ is is strictly positive, for discussion of this `cone condition' and how it can be relaxed see the final section.

For each $i\in\{1,\cdots k\}$ let $ d_i \in\mathbb R^2$ and let $T_i:\mathbb R^2\to\mathbb R^2$ be given by
\[
T_i( x):=A_i(x)+ d_i.
\]
We assume a very strong separation condition, that the maps $T_i$ map the unit disk into disjoint ellipses contained within the unit disk. This separation condition can most likely be weakened somewhat, we do not pursue this here. While the examples of section \ref{PrzSection} do not fit directly into our setting, since the associated matrices are not strictly positive, by rotating $\mathbb R^2$ they can be made to fit in the above setting.

The attractor $E$ of our iterated function system is the unique non-empty compact set satisfying 
\[
E=\bigcup_{i=1}^k T_i(E).
\]
Let
\[
T_{a_1\cdots a_n}:=T_{a_1}\circ T_{a_2}\circ \cdots T_{a_n}
\]
and
\[
E_{a_1\cdots a_n}:=T_{a_1\cdots a_n}(E)
\]
for $a_1\cdots a_n\in\{0,1\}^n$. Let $X$ denote the unit disk and let
\[
X_{a_1\cdots a_n}:= T_{a_1\cdots a_n}(X),
\]
the sets $X_{a_1\cdots a_n}$ form a sequence of nested ellipses. For each $x\in E$ there exists a unique sequence $\underline a\in\Sigma:=\{1,\cdots,k\}^{\mathbb N}$ such that
\[
\pi(\underline a):=\lim_{n\to\infty} T_{a_1\cdots a_n}(0)=x
\]
where $0$ denotes the origin. Let $\mu$ be a Bernoulli measure on $\Sigma$ with associated probabilities $p_1\cdots p_k$. By a slight abuse of notation we also denote by $\mu$ the measure $\mu\circ\pi^{-1}$ on $E$. We wish to describe the scenery flow for $\mu$.

 
The collection $A_i$ of positive matrices defines a second iterated function system on projective space. Given $A_i$, we let $\phi_i$ denote the action of $A_i^{-1}$ on $\mathbb P\mathbb R^1$, that is $\phi_i:\mathbb P\mathbb R^1\to\mathbb P\mathbb R^1$ is such that a straight line passing through the origin at angle $\theta$ is mapped to a straight line through the origin at angle $\phi_i(\theta)$ by $A_i^{-1}$. Since the matrices $A_i$ are strictly positive, the maps $\phi_i$ strictly contract the negative quadrant $\mathcal Q_2$ of $\mathbb P\mathbb R^1$.

For any $\theta\in\mathcal Q_2$ and for any sequence $\underline a\in\Sigma$ the limit
\[
\lim_{n\to\infty} \phi_{a_1}\circ\phi_{a_2}\circ\cdots\circ \phi_{a_n}(\theta)
\]
exists and is independent of $\theta$. There is a unique measure $\mu_F$ on $\mathbb P\mathbb R^1$ satisfying
\[
\mu_F(A)=\sum_{i=1}^k p_i\mu_F(\phi_i(A)).
\]
The measure $\mu_F$ is called the Furstenberg measure and has been studied for example in \cite{BPS}. 

In our example of the previous section, the Furstenberg measure is a dirac mass on direction $\frac{-\pi}{2}$ corresponding to vertical projection, and the projection of $\mu$ in this direction gave rise to a measure whose properties are key to understanding $E_{\lambda}$. In our more general case of self affine sets $E$ without a `carpet' structure, $\mu_F$ will typically have positive dimension, and the properties of projections of $\mu$ in $\mu_F$-almost every direction will be crucial.

We say that a straight line is aligned in direction $\theta$ if it makes angle $\theta$ with the positive real axis.

For $\theta\in\mathbb P\mathbb R^1$ let $\pi_{\theta}: E\to [-1,1]$ denote orthogonal projection from $E$ onto the diameter of unit disc $X$ at angle $\theta$, followed by the linear map from this diameter to $[-1,1]$. We define the projected measure $\mu_{\theta}$ on $[-1,1]$ by 
\[
\mu_{\theta}:=\mu\circ\pi_{\theta}^{-1}.
\]

{\bf Projection Condition:} We say that $\mu$ satisfies our projection condition if for $\mu_F$ almost every $\theta\in\mathbb P\mathbb R^1$ the projected measure $\mu_{\theta}$ is absolutely continuous.

In \cite{FalKemp} it was shown that the Hausdorff, box and affinity dimensions of a self-affine set coincide if the natural Gibbs measure on $E$ satisfies this projection condition. Furthermore, we gave a class of self-affine sets corresponding to an open set in parameter space for which the projection condition is satisfied, these examples were born out of the observation that the projection condition holds whenever $\dim_H \mu_F+\dim_H \mu >2$, a condition which can often be shown to hold using rough lower bounds for $\dim_H \mu$ and $\dim_H \mu_F$.

B\'ar\'any, Pollicott and Simon also gave regions of parameter space such that, for almost every set of parameters in this region, the corresponding Furstenberg measure is absolutely continuous \cite{BPS}. Assuming absolute continuity of the Furstenberg measure, our projection condition holds whenever $\dim_H \mu >1$ by Marstrand's projection theorem \cite{HuntKaloshin,Marstrand}.

\section{The Sliced Scenery Flow}\label{SlicingSection}
As a warm up to the later results describing the scenery flow for self-affine measures, we begin by considering the scenery flow on slices through self-affine measures in directions $\theta$ in the support of $\mu_F$. The results of this section do not require any projection condition.

Given $\theta\in\mathbb P\mathbb R^1$, $x\in[-1,1]$, there exists a family $\mu_{\theta,x}$ of measures defined on the slices $E_{\theta,x}:=E\cap \pi_{\theta}^{-1}(x)$ such that for each Borel set $A\subset\mathbb R^2$ we have
\[
\mu(A)=\int_{[-1,1]} \mu_{\theta,x}(A\cap E_{\theta,x}) d\mu_{\theta}(x).
\]
The family of slice measures $\mu_{\theta,x}$ is called the disintegration of $\mu$. While the above equation does not uniquely define the family of measures $\mu_{\theta,x}$, any two disintegrations of $\mu$ differ on a set of $x$ of $\mu_{\theta}$-measure $0$. See \cite{Mat} for more information on disintegration of measures.

Slicing measures can also be viewed as the limits of measures supported on thin strips around the slice. Let $E_{\theta,x,\epsilon}$ denote the strip of width $\epsilon$ around the line $E_{\theta,x}$. Then for $\mu_F$ almost every $\theta$ and $\mu_{\theta}$ almost every $x$, for any word $a_1\cdots a_n$ we have 
\begin{equation}\label{SlicingStrips}
\mu_{\theta,x}(X_{a_1\cdots a_n}\cap E_{\theta,x})=\lim_{\epsilon\to 0} \frac{\mu(X_{a_1\cdots a_n}\cap E_{\theta,x,\epsilon})}{\mu(E_{\theta,x,\epsilon})}.
\end{equation}


Let $\Sigma^{\pm}:=\{1,\cdots,k\}^{\mathbb Z}$. Given $\underline a \in \Sigma^{\pm}$ we define the angle
\[
\rho(\underline a):=\lim_{n\to\infty} \phi_{a_0}\circ \phi_{a_{-1}}\circ\cdots\circ \phi_{a_{-n}}(\theta)
\]
for any $\theta\in \mathcal Q_2$. Then let $\overline{\pi}:\Sigma^{\pm}\to \Sigma\times \mathcal Q_2$ be given by
\[
\overline{\pi}(\underline a):=(a_1a_2a_3\cdots, \rho(\underline a)).
\]

We define a map $f:\Sigma\times \mathbb P\mathbb R^1\to \Sigma\times \mathbb P\mathbb R^1$ by
\[
f(\underline a,\theta)=(\sigma(\underline a),\phi_{a_1}(\theta))
\] 
where $\sigma$ is the left shift. 
\begin{prop}
The map $f$ preserves measure $\mu\times\mu_F$. Furthermore, the system $(\Sigma\times \mathbb P\mathbb R^1,f, \mu\times\mu_F)$ is ergodic.
\end{prop}
This was proved in \cite{FalKemp}. The proof follows by observing that $\overline{\pi}$ is a continuous map which factors $(\Sigma^{\pm},\sigma,\mu)$ onto $(\Sigma\times\mathbb P\mathbb R^1,f,\mu\times\mu_F)$ and hence the ergodicity of $\sigma$ passes to the factor map $f$.

Our interest in the map $f$ stems from its relevance to scaling scenery. The following proposition is straightforward, and is proved in \cite{FalKemp}.

\begin{prop}
Let $L(\underline a,\theta)$ denote the line passing through the element of $E$ coded by $\underline a$ at angle $\theta$. Then the map $T_{a_1}^{-1}:\mathbb R^2\to \mathbb R^2$ maps the line $L(\underline a,\theta)$ to the line $L(f(\underline a,\theta))$.
\end{prop}

Our main result of this section is the following.

\begin{theorem}\label{SliceExact}
Let $\mu$ be a Bernoulli measure on a self-affine set $E\subset\mathbb R^2$ associated to strictly positive matrices $A_i$ and satisfying our separation condition. Then there exists a constant $d$ such that for $\mu_F$-almost every $\theta\in\mathbb P\mathbb R^1$ and $\mu_{\theta}$-almost every $x\in[-1,1]$ the slice measure $\mu_{\theta,x}$ is exact dimensional with dimension $d$.
\end{theorem}

The corresponding result for slices through non-overlapping self-similar measures was proved by Hochman and Shmerkin \cite{HochmanDynamics}, and this was extended to the overlapping case by Falconer and Jin \cite{FalconerJin}.

We stress again that no condition on projections of the measure $\mu$ is required in this section. The above result is an immediate corollary of the following theorem, which describes the scenery flow for the slice measures $\mu_{\theta,x}$ centered at points $\underline a \in E$ with $\pi_{\theta}(\underline a)=x$. The constant $d$ is the metric entropy of this flow. 

Let $L(\underline a,\theta,t)$ denote the line at angle $\theta$, centred at $\underline a$ and of length $e^{-t}$. Let $\mu_{\theta,\underline a,t}$ denote the measure $\mu_{\theta,\pi_{\theta}(\underline a)}$ restricted to the line $L(\underline a,\theta,t)$, linearly rescaled onto $[-1,1]$ and renormalised to have mass $1$. 

\begin{theorem}\label{SliceScaling}
There exists an ergodic fractal distribution $P$ on the space of Borel probability measures on $[-1,1]$ such that for $\mu_F$ almost every $\theta\in\mathbb P\mathbb R^1$, for $\mu$ almost every $\underline a \in \Sigma$ we have
\[
\lim_{T\to\infty}\frac{1}{T}\int_0^T \delta_{\mu_{\theta,\underline a,t}} dt \to P.
\]
\end{theorem}

Theorem \ref{SliceScaling} implies Theorem \ref{SliceExact} by a result of Hochman, see Proposition 1.19 of \cite{HochmanDynamics}. We prove Theorem \ref{SliceScaling}.
\begin{proof}
First we need to verify that the self-affinity realtion for the measures $\mu$ carries over to a corresponding relationship between the measures $\mu_{\theta,\underline a,t}$. 
Given a point $(\underline a,\theta)$ we let
\[
r_1(\underline a,\theta):=\inf\{t: L(\underline a,\theta,t)\subset X\}.
\]
Further, we let
\[
r_2(\underline a,\theta):=\inf\{t: L(\underline a,\theta,t)\subset X_{a_1}\}.
\]

Using equation \ref{SlicingStrips} and noting that
\[
T_{a_1}^{-1}(X_{a_1\cdots a_n}\cap E_{\theta,\pi_{\theta}(\underline a),\epsilon})=X_{a_2\cdots a_n}\cap E_{\phi_{a_1}(\theta),\pi_{\phi_{a_1}(\theta)}(\sigma(a)),\delta(\epsilon)}
\]
for some $\delta(\epsilon)>0$ which tends to zero as $\epsilon\to 0$, we see that
\[
\mu_{\theta,\underline a,r_2(\underline a,\theta)}=\mu_{\phi_{a_1}(\theta), \sigma(\underline a), r_1(\sigma(\underline a),\phi_{a_1}(\theta))}.
\]

The above equation says that, just as pieces of the slice through $\underline a$ at angle $\theta$ are mapped onto pieces of the slice through $\sigma(\underline a)$ at angle $\phi_{a_1}(\theta)$ by the map $T_{a_1}^{-1}$, so we can map pieces of the sliced measure onto their corresponding preimage. In particular, it allows us to understand the dynamics of zooming in on the slice measure $\mu_{\theta,\pi_{\underline a}}$ around $\underline a$ by relating small slices around $\underline a$ to larger slices around $\sigma(\underline a)$. 
We build a suspension flow that encapsulates these dynamics. 


Let roof function $r:\Sigma\times\mathbb P\mathbb R^1$ be given by $r(\underline a,\theta)=r_2(\underline a,\theta)-r_1(\underline a,\theta)$. This is the time taken to flow under $\phi$ from the line passing through $\underline a$ at angle $\theta$ and just touching the boundary of $X$ to the line centred at $\underline a,$ angle $\theta$, touching the boundary of $X_{a_1}$. 

Finally we let the flow $\psi$ be the suspension flow over the system $(\Sigma\times \mathbb P\mathbb R^1, f)$ with roof function given by $r$. That is, we define the space
\[
Z_r:=\{((\underline a,\theta),t):\underline a\in\Sigma, \theta\in \mathbb P\mathbb R^1, 0\leq t\leq r(\underline a,\theta)\}
\]
where the points $((\underline a,\theta),r(\underline a,\theta))$ and $(f(\underline a,\theta), 0)$ are identified, and let the flow $\psi_s:Z_r\to Z_r$ be given by
\[
\psi_s((\underline a,\theta),t):= ((\underline a,\theta),s+t)
\]
for $s+t\leq r(\underline a,\theta)$, extending this to a flow for all positive time $s$ by using the identification \[((\underline a,\theta), r(\underline a,\theta))=(f(\underline a,\theta), 0).\] 

We have already noted that the measure $\mu\times\mu_F$ is $f$-invariant and ergodic. This gives rise to a $\psi_s$-invariant, ergodic measure $\nu$ on $Z_r$ given by
\[
\nu=(\mu\times\mu_F\times\mathcal L)|_{Z_r}
\]
where $\mathcal L$ denotes Lebesgue measure.

There is an obvious factor map $F$ from $Z_r$ to the space of Borel probability measures on $[-1,1]$ given by letting

\[
F((\underline a,\theta),t)=\mu_{\underline a,\theta,t}.
\]

We have
\[
F(\psi_s(\underline a,\theta,t))=\mu_{\underline a,\theta,t+s}
\]
and thus we see that for $\mu\times \mu_F$ almost every pair $(\underline a,\theta)$ we have that the scenery flow on the measure $\mu_{\theta,\underline a,1}$ generates the ergodic fractal distribution $P=\nu\circ F^{-1}$. This completes the proof of Theorem \ref{SliceScaling} and hence of Theorem \ref{SliceExact}.



\end{proof}

In essence, one can combine the work of sections \ref{PrzSection} and \ref{SlicingSection} to give all the intuition needed to describe the scenery flow for the self-affine measures we consider. What follows, which is occasionaly quite technical, verifies that this intuition is correct.



\section{Dilating Ellipses and a Related Flow}

Before describing the scenery flow, we describe a map from the space of measures on large ellipses to probability measures on $X$. This map plays the role of the map $D$ of Observation 2, and will allow us to approximate the scenery flow on $\mu$ arbitrarily well.

Given $\underline a\in\Sigma,$ $\theta\in\mathbb P\mathbb R^1$, $r_1, r_2>0$ we let the ellipse $Y_{\underline a,\theta,r_1,r_2}$ be the ellipse centred at $\pi(\underline a)$, with long axis of length $2e^{-r_1}$ aligned in direction $\theta$ and with short axis of length $2e^{-r_2}$.

Given $(\underline a,\theta,r_1,r_2)$ such that $Y_{\underline a,\theta,r_1,r_2}\not\subset X_{a_1}$, we let $D_{\underline a,\theta,r_1,r_2}:Y_{\underline a,\theta,r_1,r_2}\to X$ be the bijection which maps the major axis of $Y_{\underline a,\theta,r_1,r_2}$ to $\{0\}\times[-1,1]$ and the minor axis of $Y_{\underline a,\theta,r_1,r_2}$ to $[-1,1]\times\{0\}$.

Let $D_{\underline a,\theta,r_1,r_2}$ also denote the analagous map which maps finite measures on ellipses $Y_{\underline a,\theta,r_1,r_2}$ to probability measures on $X$. As in observation $3$ of Section \ref{PrzSection}, we consider what happens to the family of dilated measures as the minor axis of an ellipse shrinks. 

\begin{lemma}\label{DilatingLemma}
Let $\underline a\in\Sigma, \theta\in\mathbb P\mathbb R^1, r_1>0$. Suppose that for $\mu_{\underline a,\theta,r_1}$ almost every $\underline b\in\Sigma$ we have that there exist infinitely many $n\in\mathbb N$ such that the projection of $\mu|_{X_{b_1\cdots b_n}}$ in direction $\theta$ is absolutely continuous, and that the scenery flow on this projected measure centred at $\pi_{\theta}(\underline b)$ converges to Lebesgue measure. Then we have
\[
\lim_{r_2\to \infty} D_{\underline a, \theta, r_1,r_2}(\mu|_{Y(\underline a, \theta, r_1,r_2)})=\frac{(\mathcal L\times \mu_{\underline a,\theta,r_1})|_{X}}{(\mathcal L\times \mu_{\underline a,\theta,r_1})(X)}.
\]
\end{lemma}

\begin{proof}
Our notion of convergence here is that of the Prokhorov metric. It is enough to show that, for all $N\in\mathbb N$, we can divide the unit square into a grid of $2N+1$ squares $A_{i,j}$ of equal side length and have that for each $i,j\in\{-N,\cdots,N\}$ such that $A_{i,j}\subset X$,
\[
 \left(D_{\underline a, \theta, r_1,r_2}(\mu|_{Y(\underline a, \theta, r_1,r_2)})\right)(A_{i,j})\to\frac{(\mathcal L\times \mu_{\underline a,\theta,r_1})|_{X}}{(\mathcal L\times \mu_{\underline a,\theta,r_1})(X)}(A_{i,j})
\]

First we consider squares $A_{0,j}$ whose $x$-coordinate is at the origin. Since slicing measures are almost surely the limit of the measures $\mu$ restricted to a thin strip around the slice, we have that the relative distribution of mass within the squares $A_{0,j}$ converges to the slicing measure $\mu_{\underline a,\theta,r_1}$ as $r_2\to \infty$ (given $\theta$ this holds for $\mu$-almost every $\underline a$.

Now we fix $j$ and consider the horizontal distribution of mass in \[
\frac{\mu|_{Y(\underline a, \theta, r_1,r_2)}\left( D_{\underline a, \theta, r_1,r_2}^{-1}(A_{i,j})\right)}{\mu(Y(\underline a,\theta,r_1,r_2))}\] for $i$ varying. 

We note that $Y(\underline a, \theta, r_1,r_2)$ intersects various ellipses. The ellipses $X_{b_1\cdots b_m}$ are separated, and as $r_2\to \infty$ the angle of the strip
\[
D_{\underline a, \theta, r_1,r_2}(X_{b_1\cdots b_m}\cap Y(\underline a,\theta,r_1,r_2))
\]
tends to the horizontal (indeed, any line which is not in direction $\theta$ gets pulled towards the horizontal by $D_{\underline a, \theta, r_1,r_2}$, and as $r_2\to \infty$ this effect becomes ever more pronounced). For $m$ large enough, each strip $D_{\underline a, \theta, r_1,r_2}(X_{b_1\cdots b_m}\cap Y(\underline a,\theta,r_1,r_2))$ is contained within the horizontal rectangle $\cup_{i=-N}^N A_{i,j}$ for some $j\in\{-N,\cdots,N\}$.
\begin{figure}[h]
\centering
\includegraphics[width=90mm,angle=90]{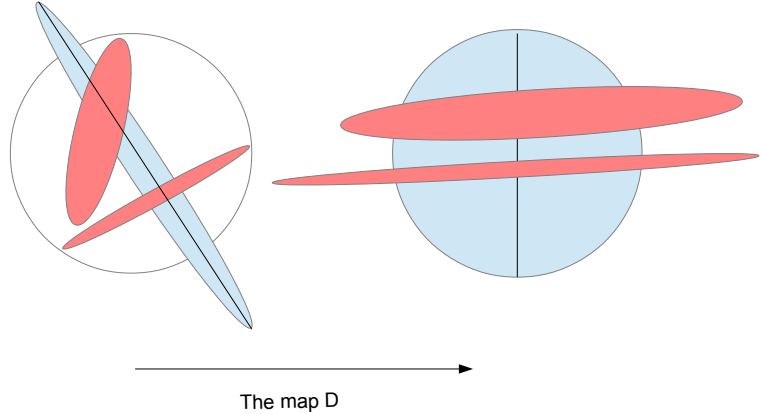}
\caption{The action of $D$ on an ellipse $Y$.}
\end{figure}


We want to understand the distribution of mass horizontally within the rectangles $\cup_{i=-N}^N A_{i,j}$. 

We will consider the projection of $\mu|_{X_{b_1\cdots b_m}}$ in direction $\theta$, intersected with $Y(\underline a,\theta,r_1,r_2)$, dilated to be supported on $[-1,1]$ and normalised to have mass $1$. This is just the scenery flow on the projection of $\mu|_{X_{b_1\cdots b_m}}$ in direction $\theta$, centred at $\pi_{\theta}(\underline b)$ at time $r_2$. 

Now we assumed that $\pi_{\theta}(\mu|_{X_{b_1\cdots b_m}})$ was absolutely continuous with positive density at $\pi_{\theta}(\underline b)$. Then by lemma \ref{ACFlow} this scenery flow converges to Lebesgue measure. Thus the horizontal distribution of mass within the rectangles $\cup_{i=-m(j)}^m(j) A_{i,j}$ converge to Lebesgue measure as $r_2\to \infty$ for all $j$, where $m(j)$ is the largest natural number such that $A_{m(j),j}\subset X$. Then we are done. 
\end{proof}
This yields the following corollary.
\begin{cor}\label{CorCondition}
Suppose that our projection condition holds, i.e. that the projected measure $\mu_{\theta}$ is absolutely continuous for $\mu_F$ almost every $\theta\in \mathbb P\mathbb R^1$. Then for $\mu$-almost every $\underline a$, $\mu_F$ almost every $\theta$ and all $r_1$ we have
\[
\lim_{r_2\to \infty}D_{\underline a, \theta, r_1,r_2}(\mu|_{Y(\underline a, \theta, r_1,r_2)})=\frac{(\mathcal L\times \mu_{\underline a,\theta,r_1})|_{X}}{(\mathcal L\times \mu_{\underline a,\theta,r_1})(X)}.
\]
\end{cor}

\begin{proof}
First we note, using our affinity relation, that the measure obtained by projecting $\mu|_{X_{b_1\cdots b_n}}$ in direction $\theta\in supp(\mu_F)$ centred at $\pi_{\theta}(\underline b)$ is a scaled down copy of the measure obtained by projecting $\mu$ in direction $\phi_{b_n}\circ\cdots \circ\phi_{b_1}(\theta)$ centred at $\pi_{\phi_{b_n}\circ\cdots \circ\phi_{b_1}(\theta)}(\sigma^n(\underline b))$, see \cite{FalKemp} for a careful proof. 

Since the directions $\theta$ are distibuted according to $\mu_F$, for $\mu$-almost every $\underline b$ it follows that the projected measure $\mu_{\phi_{b_n}\circ\cdots \circ\phi_{b_1}(\theta)}$ is absolutely continuous for all $n$ and that the scenery flow centred at $\pi_{\phi_{b_n}\circ\cdots \circ\phi_{b_1}(\theta)}(\sigma^n(\underline b))$ converges to Lebesgue measure. 

If a condition holds for $\mu\times \mu_F$-almost every $(\underline b,\theta)$ then it follows that for all $r_1>0$, for $\mu\times\mu_F$ almost every $(\underline a,\theta)$ the condition holds for $\mu_{\underline a,\theta,r_1}$ almost every $\underline b$. Then we see that the hypotheses of Lemma \ref{DilatingLemma} hold, and so the conclusions hold also, as required.


\end{proof}

\subsection{An Ergodic Flow}
Given a pair $(\underline a, t)$ we let 
\[
n=n(\underline a,t)=\max\{n\in\mathbb N: B(\underline a,e^{-t})\subset X_{a_1\cdots a_n}\}.
\]
Then we associate to small ball $B(\underline a,e^{-t})$, coupled with angle $\theta$, the ellipse 
\[Y_{\sigma^n(\underline a), \phi_{a_n}\circ\cdots\circ \phi_{a_1}(\theta), t+\log(\alpha_2(a_1\cdots a_n)),t+\log(\alpha_1(a_1\cdots a_n))}.\]

Note that the first two parameters here are equal to $f^n(\underline a,\theta)$. Since $\alpha_2(a_1\cdots a_n)\to 0$ as $t,n\to\infty$, $\log(\alpha_2(a_1\cdots a_n))$ is negative. In fact, the quantity $t+\log(\alpha_2(a_1\cdots a_n))$ remains bounded as $t\to\infty$.

Let 
\[
\nu(\underline a,\theta,t):= D_{f^n(\underline a,\theta),t+\log(\alpha_1(a_1\cdots a_n)),t+\log(\alpha_2(a_1\cdots a_n))}(\mu|_{Y_{f^n(\underline a,\theta). t+\log(\alpha_1(a_1\cdots a_n)),t+\log(\alpha_2(a_1\cdots a_n))}}),
\]
The measures $\nu(\underline a,\theta,t)$ are elements of $\mathcal M$.

We define a map $F_2$ which takes probability measures on $[-1,1]$ to probability measures on $X$ by

\[
F_2(m):= \frac{(\mathcal L\times m)|_{X}}{(\mathcal L\times m)(X)}.
\]

Let the distribution $ P_2$ on the space of Borel probability measures on $X$  be the image of $ P$ under $F_2$, where $ P$ was defined in Section \ref{SlicingSection}. Then the two dimensional scenery flow on $(\mathcal M, P_2)$ is a factor of the one dimensional scenery flow on $(\mathcal M_1, P)$ under the factor map $F_2$, and so it follows immediately that $ P_2$ is an ergodic fractal distribution, since ergodicity passes to factors of ergodic systems. One can readily verify that the distribution $ P_2$ is quasi-Palm.

\begin{theorem}
For $\mu\times \mu_F$ almost every pair $(\underline a,\theta)$ 
\[
\lim_{T\to\infty}\frac{1}{T}\int_0^T \nu(\underline a,\theta,t) dt = P_2.
\]
\end{theorem}

\begin{proof}
Let $(\underline a,\theta)$ be such that the sliced scenery flow on the measure $\mu_{\underline a,\theta,1}$ generates $ P$, and such that $f^n(\underline a,\theta)$ satisfies the conditions of Corollary \ref{CorCondition} for each $n\in\mathbb N$. The set of $(\underline a,\theta)$ for which this holds has $\mu\times\mu_F$ measure one, since it is a countable intersection of sets of measure $1$.

For $\epsilon,N>0$, let the bad set $B(\epsilon,N)$ be given by
\[
B(\epsilon,N):=\left\{(\underline a,\theta):\left|D_{\underline a, \theta, r_1,r_2}(\mu|_{Y(\underline a, \theta, 0,r_2)})-\frac{(\mathcal L\times \mu_{\underline a,\theta,0})|_{X}}{(\mathcal L\times \mu_{\underline a,\theta,0})(X)}\right|>\epsilon \text{ for some } r_2>N\right\}.
\]
Then for all $\epsilon,\epsilon_2>0$, using Egorov's theorem and Lemma \ref{DilatingLemma}, there exist $N>0$ such that
\[
(\mu\times\mu_F)(B(\epsilon,\delta))<\epsilon_2.
\]
Now note that for any $\underline a$ there exists a $T$ such that $T_{a_1\cdots a_n}^{-1}(B(\pi(\underline a),e^{-t}))$ is an ellipse with minor axis of length less than $e^{-N}$ for all $t>T$. Then since the scenery flow on $(\underline a,\theta,1)$ generates $ P$, and since $\epsilon_1, N$ were arbitrary, we see that
\[
\lim_{T\to\infty}\frac{1}{T}\int_0^T \nu(\underline a,\theta,t) dt = P_2.
\]
as required.
\end{proof}

Finally we state a continuity result. The proof of this result requires a little geometry, and is most likely of limited interest, and so can be found in the appendix.

\begin{prop}\label{ContinuityProp}
For each $\underline a\in\Sigma, t\in\mathbb R$ the map \[\theta\to \nu_{\underline a,\theta,t}\] is continuous in $\theta$ and this continuity is uniform over $t$. in particular, for all $\theta\in\mathbb P\mathbb R^1$ and for all $\epsilon>0$ there exists $\delta>0$ such that if $|\theta-\theta'|<\delta$ then for all subsets $A\subset\mathcal M$ we have
\[
\left|\lim_{T\to\infty}\frac{1}{T}\mathcal L\{t\in[0,T]:\nu(a,\theta,t)\in A\}-\lim_{T\to\infty}\frac{1}{T}\mathcal L \{t\in[0,T]:\nu(\underline a,\theta',t)\in A\}\right|<\epsilon,
\]
and so the sceneries generated by $\nu(\underline a,\theta,t)$ and $\nu(\underline a,\theta',t)$ are close.
\end{prop}

\section{The Full Scenery Flow}
We now relate the scenery flow on $\mu$ to the measures $\nu$ of the previous section.

We let $S_{t,\underline a}$ denote the bijective linear map from $B(\pi(\underline a),e^{-t})$ to $X$ given by expanding all vectors by $e^t$ and translating the resulting ball to the origin. We also let $S_{t,\underline a}$ be the scenery flow map from finite measures on $B(\pi(\underline a),e^{-t})$ to probability measures on $X$. In this section we first describe the preimages of small balls under maps $T_{a_1\cdots a_n}^{-1}$, and then decompose the scenery flow for $\mu$ using the maps $D$ of the previous section.

The fact that we are considering only strictly positive matrices leads to some simple observations about the intersection of $B(x,r)$ with the self-affine set $E$. Let $\alpha_1(a_1\cdots a_n)$, $\alpha_2(a_1\cdots a_n)$ denote the lengths of the major and minor axes of the ellipse $X_{a_1\cdots a_n}$. Then the ratio $\frac{\alpha_2(a_1\cdots a_n)}{\alpha_1(a_1\cdots a_n)}$ tends to $0$ as $n\to\infty$ at some uniform rate independent of $\underline a$ (see \cite{FalKemp}). 

There exists a H\"older continuous function $F:\Sigma\to\mathbb P\mathbb R^1$ such that, for each $\underline a\in\Sigma$, the ellipses $X_{a_1\cdots a_n}$ are aligned so that the angle that their long axis makes with the $x$-axis tends to $ F(\underline a)$ as $n\to\infty$. This convergence is uniform over $\underline a\in\Sigma$. In fact, $ F(\underline a)$ is given by
\[
F(\underline a):=\lim_{n\to\infty} \phi_{a_1}^{-1}\circ\cdots\circ\phi_{a_n}^{-1}(0)\in\mathcal Q_1.
\]
The strong stable foliation, which gives the limiting direction of the minor axis of ellipses $X_{a_1\cdots a_n}$, is given by
\[
F_{ss}(\underline a):=\lim_{n\to\infty} \phi_{a_1}\circ\cdots\circ \phi_{a_n}(0)\in\mathcal Q_2.
\]
Note that $ F(\underline a)$ and $ F_{ss}(\underline a)$ are perpendicular.

Let $\theta(a_1\cdots a_n)\in\mathcal Q_2$  be the direction of the minor axis of the ellipse $X_{a_1\cdots a_n}$. 

\begin{prop}\label{PullbackProp}
Let $e^{-t}<\alpha_2(a_1\cdots a_n)$. Then
\[
T_{a_1\cdots a_n}^{-1}(B(\pi(\underline a),e^{-t}))
\]
is an ellipse centred at $\pi(\sigma^n(\underline a))$ with major axis of length $e^{-t}.(\alpha_2(a_1\cdots a_n))^{-1}$ aligned in direction
\[
\phi_{a_n}\circ\cdots \phi_{a_1}(\theta(a_1\cdots a_n))
\]
and minor axis of length equal to $e^{-t}.(\alpha_1(a_1\cdots a_n))^{-1}$. 
\end{prop}
Stated using our notation for ellipses, this says
\[
T_{a_1\cdots a_n}^{-1}(B(\pi(\underline a),e^{-t}))=Y_{\sigma^n(\underline a),\phi_{a_n}\circ\cdots \phi_{a_1}(\theta(a_1\cdots a_n)), t+\log(\alpha_2(a_1\cdots a_n)), t+\log(\alpha_1(a_1\cdots a_n))}.
\]
\begin{proof}
Lines which bisect the ellipse $X_{a_1\cdots a_n}$ just touching the edges and passing through the centre are mapped by $T_{a_1\cdots a_n}^{-1}$ to lines passing through the origin which just touch the boundary of the unit disk. This fact allows us to see how much the linear map $T_{a_1\cdots a_n}^{-1}$ expands different lines.

In particular, the maximal expansion rate is on lines in direction $\theta(a_1\cdots a_n)$, parallel to the minor axis of $X_{a_1\cdots a_n}$. These are expanded linearly by a factor $\frac{1}{\alpha_2(a_1\cdots a_n)}$,  and by the definition of $\phi_i$ we see they are mapped to direction $\phi_{a_n}\circ\cdots \phi_{a_1}(\theta(a_1\cdots a_n))$, note the reversed order of the word $a_1\cdots a_n$ here.

The major axis of $X_{a_1\cdots a_n}$ gives rise to the smallest expansion rate of the map $T_{a_1\cdots a_n}^{-1}$, which is $\frac{1}{\alpha_1(a_1\cdots a_n)}$, thus the minor axis of $T_{a_1\cdots a_n}^{-1}B(\pi(\underline a),e^{-t}))$ has length $e^{-t}.(\alpha_1(a_1\cdots a_n))^{-1}$.
\end{proof}

We now discuss functions which map our ellipses $T_{a_1\cdots a_n}^{-1}(B(\pi(\underline a),r))$ to the unit disk. Note that any bijective linear map from $B(\pi(\underline a),r)$ to $X$ which maps $\pi(\underline a)$ to the origin and which preserves the directions $ F(\underline a)$ and $ F_{ss}(\underline a)$ must be the same as our dilation map $S_{-\log r,\underline a}$. This is because a linear map in $\mathbb R^2$ is uniquely determined by its action on any two vectors which span $\mathbb R^2$.


\begin{prop}\label{DecompositionProp}
We have
\[
S_{t,\underline a}= R_{\theta^{\perp}(a_1\cdots a_n)}\circ D_{\sigma^n(\underline a),\phi_{a_n}\circ\cdots \phi_{a_1}(\theta(a_1\cdots a_n)), t+\log(\alpha_2(a_1\cdots a_n)), t+\log(\alpha_1(a_1\cdots a_n)} \circ T_{a_1\cdots a_n}^{-1}
\]

where $n=n(\underline a,t)$ is such that $n$ is the largest natural number for which $B(\pi(\underline a),e^{-t})\subset X_{a_1\cdots a_n}$.
\end{prop}

\begin{proof}
The previous proposition noted that $T_{a_1\cdots a_n}^{-1}(B(\pi(\underline a),e^{-t}))$ is an ellipse centred at $\pi(\sigma^n(\underline a))$. It also follows from the proof that $T_{a_1\cdots a_n}^{-1}$ maps lines in direction $\theta(a_1\cdots a_n)$ to lines in direction $\phi_{a_n}\circ\cdots \phi_{a_1}(\theta(a_1\cdots a_n))$. Furthermore, the perpendicular angles of the major and minor axis of the ellipse $X_{a_1\cdots a_n}$ are mapped on to the perpendicular angles of the minor and major axis of the ellipse $T_{a_1\cdots a_n}^{-1}(B(\pi(\underline a,e^{-t})))$.

Then
\[
D_{\sigma^n(\underline a),\phi_{a_n}\circ\cdots \phi_{a_1}(\theta(a_1\cdots a_n)), t+\log(\alpha_2(a_1\cdots a_n)), t+\log(\alpha_1(a_1\cdots a_n)} \circ T_{a_1\cdots a_n}^{-1}
\]
maps $B(\underline a,e^{-t})$ bijectively onto $X$, where the diameter of $B(\underline a,e^{-t})$ at angle $(\theta(a_1\cdots a_n))$ is mapped to $\{0\}\times[-1,1]$ and the diameter at angle $\theta^{\perp}(a_1\cdots a_n)$ is mapped to $[-1,1]\times\{0\}$.

Rotating by angle $\theta^{\perp}(a_1\cdots a_n)$ we see that the image of the major and minor axes of $X_{a_1\cdots a_n}$ are oriented in the correct direction. 

Then we see that our map is a bijective map from $B(\pi(\underline a),e^{-t})$ to $X$ which maintains the directions $\theta(a_1\cdots a_n)$ and $\theta^{\perp}(a_1\cdots a_n)$, so we are done.
\end{proof}

In particular, this yields the following theorem.

\begin{theorem}\label{MainTheorem}
Let $\mu$ be a Bernoulli measure on a self-affine set $E$ associated to strictly positive matrices, and assume that for $\mu_F$ almost every $\theta\in\mathbb P\mathbb R^1$ the image $\mu_{\theta}$ of $\mu$ under projection in direction $\theta$ is absolutely continuous. Then for $\mu$-almost every $\underline a$ the scenery flow $S_{t,\underline a}(\mu)$ is given by
\[
S_{t,\underline a}(\mu)=R_{\theta^{\perp}(a_1\cdots a_n)}(\nu_{\underline a,\theta(a_1\cdots a_n),t}).
\]
As $t\to\infty$ this flow is asymptotic to the flow
\[
\mathcal R_{F(\underline a)}(\nu_{\underline a,  F_{ss}(\underline a),t})
\]
and so generates the ergodic fractal distribution $R_{ F(\underline a)}\circ P_2$. 
\end{theorem}

By $R_{ F(\underline a)}\circ P_2$ we mean the distribution on $\mathcal M$ obtained by picking measures $\mu\in\mathcal M$ according to $ P_2$ and then rotating the resulting measure by angle $R_{ F(\underline a)}$.

Hence we see that $\mu$ is not a uniformly scaling measure, unless the foliation $ F(\underline a)$ gives the same angle for each $\underline a$. This happends only when the maps $\phi_i$ all have a common fixed point, in which case the Furstenberg measure $\mu_F$ is a Dirac mass and the corresponding self-affine set has a carpet like construction. In particular, Theorem \ref{PrzTheorem} is a corollary to this theorem.

Finally we comment that one does not automatically have that for $\mu$ almost every $\underline a$ the flow $\nu_{\underline a,  F_{ss}(\underline a),t}$ equidistributes with respect to $P_2$, since there is an obvious dependence between $\underline a$ and $F_{ss}(\underline a)$. Here we rely on our continuity proposition (Proposition \ref{ContinuityProp}) which allows us to replace $F_{ss}(\underline a)$ with $\mu_F$-typical angles $\theta$ close to $F_{ss}(\underline a)$ such that the distance between the orbits $\nu_{\underline a,\theta,t}$ and $\nu_{\underline a,  F_{ss}(\underline a),t}$ remains small. 

\begin{proof}
First we note that, by the ergodic theorem, for $\mu$ almost every $\underline a$ and for all $\epsilon>0$ there exists $\theta\in( F_{ss}(\underline a)-\delta,  F_{ss}(\underline a)+\delta)$ such that the family of measures $\nu_{\underline a,\theta,t}$ equidistributes with respect to $ P_2$. Now since $\theta(a_1\cdots a_n)\to  F_{ss}(\underline a)$ we see that the sequence $\theta(a_1\cdots a_n)$ is eventually bounded within distance $2\delta$ of $\theta$. Then by Proposition \ref{ContinuityProp} we have that the measures $\nu_{\underline a,\theta(a_1\cdots a_n),t}$ and $\nu_{\underline a,\theta,t}$ are within $\epsilon$ of each other, and so, since $\epsilon$ was arbitrary, we have that the family of measures $\nu_{\underline a,\theta(a_1\cdots a_n),t}$ generate $ P_2$.

Finally, incorporating the rotation element and using Proposition \ref{DecompositionProp} we have that $S_{t,\underline a}(\mu)$ generates the distribution $R_{ F(\underline a)}\circ P_2$.
\end{proof}

\section{Further Comments and Open Problems}
Despite having been worked on for over 25 years, a general theory of the dimension of self-affine sets has proved ellusive. Indeed, questions such as whether box dimension always exists for self-affine sets remain open. The scenery flow seems like a natural tool to transfer results from ergodic theory to the study of dimension for self-affine sets.

There are a number of further questions which could lead towards a more general theory of scenery flow for self-affine sets.

{\bf Question 1:} Can one conclude that examples of section $2$ uniformly scaling measures generating ergodic fractal distributions whenever the corresponding Bernoulli convolution is a uniformly scaling measure generating an ergodic fractal distribution?

{\bf Question 2:} Are overlapping self-similar sets uniformly scaling measures generating ergodic fractal distributions? What about projections of self-affine sets? The second part will most likely follow from the first, given the dynamical structure of projections of self-affine sets described in \cite{FalKemp}.

{\bf Question 3:} Suppose that for $\mu_F$ almost every $\theta$ the projection $\pi_{\theta}:E\to[-1,1]$ is one to one. Can one conclude that the self affine measure $\mu$ is a uniformly scaling measure generating an ergodic fractal distribution?

Finally we comment on the condition that the matrices generating our self-affine set should be strictly positive. This condition ensures that the maps $\phi_i$ strictly contract the negative quadrant and hence that $\mu_F$ can be defined via an iterated function system construction. The condition is also useful in making a lot of convergence results uniform. It seems likely that the condition can be relaxed. The Furstenberg measure can be defined without any cone condition, see \cite{BenoistQuint}.

\section{Appendix: Continuity in $\theta$}

\begin{lemma}\label{AngEll1}
\[
\frac{\alpha_1(a_1\cdots a_n)}{\alpha_2(a_1\cdots a_n)}\tan(\phi_{a_1\cdots a_n}(\theta)-\phi_{a_1\cdots a_n}(\theta(a_1\cdots a_n)))=  \tan(\theta-\theta(a_1\cdots a_n))
\]
\end{lemma}
\begin{proof}
The linear map $A_{a_1\cdots a_n}^{-1}$ stretches lines at angle $\theta(a_1\cdots a_n)$ by $\alpha_2(a_1\cdots a_n)^{-1}$ and lines at angle $\theta(a_1\cdots a_n)^{\perp}$ by $\alpha_1(a_1\cdots a_n)^{-1}$. The lemma follows using basic geometry.
\end{proof}

We now consider when one ellipse can fit inside an expanded, rotated concentric copy of itself.

\begin{lemma}\label{AngEll2}
Let $Y$ be an ellipse centred at the origin with major and minor axes of length $\alpha_1, \alpha_2$ respectively and with major axis oriented along the $y$-axis. Let $Z$ be an ellipse centred at the origin with major and minor axes of length $(1-\epsilon)\alpha_1, (1-\epsilon)\alpha_2$ respectively with major axis oriented at angle $\theta$ from the vertical. Then $Z\subset Y$ whenever $\frac{\alpha_1}{\alpha_2}\tan(\theta)<\frac{1}{1-\epsilon}-1$. 
\end{lemma}

\begin{proof}
The line from the origin to the boundary of $Y$ at angle $\rho$ has length $\alpha_1 \cos(\rho) + \alpha_2 \sin(\rho)$. The corresponding line for ellipse $Z$ has length
\begin{eqnarray*}
& &(1-\epsilon)(\alpha_1(\cos(\rho-\theta))+\alpha_2(\sin(\rho-\theta)))\\
&=&(1-\epsilon)(\alpha_1(\cos(\rho)\cos(\theta)+\sin(\rho)\sin(\theta))+\alpha_2(\sin(\rho)\cos(\theta)-\cos(\rho)\sin(\theta)))\\
&\leq& (1-\epsilon)(\alpha_1\cos(\rho)+\alpha_2\sin(\rho))(\cos(\theta)+\frac{\alpha_1}{\alpha_2}(\sin(\theta)). 
\end{eqnarray*}
So if we have
\[
\cos(\theta)+\frac{\alpha_1}{\alpha_2}\sin(\theta)\leq\frac{1}{1-\epsilon}
\]
then we will have for each angle $\rho$ that the slice through $Z$ at angle $\rho$ is a subset of the slice through $Y$ at angle $\rho$, and hence that $Z\subset Y$. The above inequality holds whenever
\[
\frac{\alpha_1}{\alpha_2}\tan(\theta)<\frac{1}{1-\epsilon}-1
\]
as required.
\end{proof}

Combining the last two lemmas gives us the following lemma.

\begin{lemma}
let $\underline a\in\Sigma$ and suppose that $\theta$ is such that $|\tan(\theta)-\tan(F_{ss}(\underline a))|<\frac{1}{1-\epsilon}-1$. Then
\[
Y_{f^n(\underline a,\theta), -\log(\alpha_1(a_1\cdots a_n)-\epsilon, -\log(\alpha_1(a_1\cdots a_n))-\epsilon}\subset Y_{f^n(\underline a, \theta(a_1\cdots a_n)),-\log(\alpha_1(a_1\cdots a_n)), -\log(\alpha_2(a_1\cdots a_n)))}
\]
for all large enough $n$.
\end{lemma}
\begin{proof}
By lemma \ref{AngEll1} we have that
\[
\frac{\alpha_1(a_1\cdots a_n)}{\alpha_2(a_1\cdots a_n)}\tan(\phi_{a_1\cdots a_n}(\theta)-\phi_{a_1\cdots a_n}(\theta(a_1\cdots a_n)))=  \tan(\theta-\theta(a_1\cdots a_n))<\frac{1}{1-\epsilon}-1
\]
eventually, since $\theta(a_1\cdots a_n)\to  F_{ss}(\underline a)$. Then by lemma \ref{AngEll2} we are done.
\end{proof}
We now consider our maps $D$ which dilate ellipses. We show that if $Z\subset Y$ with the area of $Y$ close to that of $Z$ then the measure $D_Z(\mu|_{Z})$ is close to $D_Y(\mu|_{Y})$. We do this by showing that the natural magnification map $D_Z$ from $Z$ to the unit disk is the same as first magnifying $Z$ using the magnification map $D_Y$ on $Y$ to get some other ellipse $W\subset X$, and then using the magnification map $D_W$ on $W$.
\begin{lemma}
Let $Y_{\underline a,\theta,r_1,r_2}\subset Y_{\underline a,\theta',r_1',r_2'}.$ Let $\underline a'',\theta'',r_1'',r_2''$ be such that
\[
D_{\underline a,\theta',r_1',r_2'}(Y_{\underline a,\theta,r_1,r_2})=Y_{\underline a'',\theta'',r_1'',r_2''}\subset X.
\]

Then
\[
D_{\underline a,\theta,r_1,r_2}=D_{\underline a'',\theta'',r_1'',r_2''}\circ D_{\underline a,\theta',r_1',r_2'}:Y_{\underline a,\theta,r_1,r_2}\to X.
\]
\end{lemma}

\begin{proof}
Since the map $D_{\underline a'',\theta'',r_1'',r_2''}\circ D_{\underline a,\theta',r_1',r_2'}:Y_{\underline a,\theta,r_1,r_2}\to X$ is bijective and maps the major and minor axes onto the vertical and horizontal axes, this is immediate.
\end{proof}

\begin{lemma}
For all $\epsilon>0$ there exists $\delta>0$ such that whenever ellipse $W\subset X$ has area larger than $1-\delta$ and long axis oriented within $\delta$ of the vertical then the map $D_W:W\to Z$ is within $\epsilon$ of the identity map.
\end{lemma}
This is again immediate.

Putting all of the previous lemmas together yields the following theorem.
\begin{theorem}
For all $\epsilon>0$ there exists $\delta>0$  such that, for all $\theta$ with $| F_{ss}(\underline a)-\theta|<\delta$ we have that
\[
d\left(\nu_{\underline a,\theta,t},D_{\sigma^n(\underline a),\phi_{a_n}\circ\cdots \phi_{a_1}(\theta(a_1\cdots a_n)), t+\log(\alpha_2(a_1\cdots a_n)), t+\log(\alpha_1(a_1\cdots a_n))} \circ T_{a_1\cdots a_n}^{-1}(\mu|_{B(\pi(\underline a), e^{-t})})\right)<\epsilon.
\]
\end{theorem}

This continuity theorem allows one to use the flow giving rise to measures $\nu$ to infer properties of the scenery flow.

\section*{Acknowledgements}
Many thanks to Jon Fraser and Kenneth Falconer for many useful discussions. This work was supported by the EPSRC, grant number EP/K029061/1.

\bibliographystyle{plain} 
\bibliography{SelfAffine.bib}
\end{document}